\documentclass[11pt]{amsart}

\usepackage{amsmath, amssymb, amsthm, verbatim, hyperref}
\newtheorem{thm}{Theorem}[section]
\newtheorem{lemma}[thm]{Lemma}
\newtheorem{prop}[thm]{Proposition}
\newtheorem{crl}[thm]{Corollary}

\theoremstyle{definition}
\newtheorem{dfn}[thm]{Definition}

\newtheorem{rem}[thm]{Remark}

\newcommand{\rx}{\mathbb{R}[\underline{X}]}
\newcommand{\rfps}{\mathbb{R}[[\underline{X}]]}
\newcommand{\ux}{\underline{X}}

\newcommand{\sos}{\sum\mathbb{R}[\underline{X}]^2}

\newcommand{\reals}{\mathbb{R}}
\newcommand{\naturals}{\mathbb{N}}
\newcommand{\Pos}{\mbox{Psd}}
\newcommand{\la}{\langle}
\newcommand{\ra}{\rangle}
\newcommand{\supp}{\text{supp}}

\begin{document}
\title[Moment Problem Continuous Positive Semidefinite functionals]{The Moment Problem for Continuous Positive Semidefinite Linear functionals}
\author[M. Ghasemi]{Mehdi Ghasemi$^{*}$}
\author[S. Kuhlmann]{Salma Kuhlmann$^{**}$}
\author[E. Samei]{Ebrahim Samei$^{*1}$}
\date{\today}
\thanks{$^1$The third author was partially supported by an NSERC discovery grant.}
\address{$^{*}$Department of Mathematics and Statistics,\newline
University of Saskatchewan,\newline
Saskatoon, SK S7N 5E6, Canada}
\address{$^{**}$Fachbereich Mathematik und Statistik,\newline
Universit\~{a}t Konstanz\newline
78457 Konstanz, Germany}
\email{mehdi.ghasemi@usask.ca}
\email{salma.kuhlmann@uni-konstanz.de}
\email{samei@math.usask.ca}
\begin{abstract}
Let $\tau$ be a locally convex topology on the countable dimensional polynomial $\reals$-algebra $\rx:=\reals[X_1,\ldots,X_n]$. Let $K$ be a closed subset of
$\reals^n$, and let $M:=M_{\{g_1, \cdots g_s\}}$ be a finitely generated quadratic module in $\rx$.
We investigate the following question: When is the cone $\Pos(K)$ (of polynomials
nonnegative on $K$) included in the closure of $M$? We give an interpretation of
this inclusion with respect to representing continuous linear functionals by measures.
We discuss several examples; we compute the closure of $M=\sos$ with respect to weighted norm-$p$ topologies.
We show that this closure coincides with the cone $\Pos(K)$ where $K$ is a certain convex compact polyhedron.
\end{abstract}
\maketitle
\section{Introduction}
Given a finite subset $S:=\{ g_1 ,\cdots, g_s \}$ of the polynomial
ring $\rx$, the question of approximating polynomials nonnegative on
the basic closed semialgebraic set $K_S$  via elements of the quadratic module $M_S$ (see Definition
\ref{preordering}) is a main topic in real algebraic geometry and
has many applications in optimization and functional analysis (see
Lasserre \cite{JL-OptMom}). Putinar's Archimedean Positivstellensatz
\cite{Pu} 
brought important improvements to the Positivstellensatz \cite{Kr}:
for $K$ a {\it compact} basic closed semialgebraic set, and $S$ any
representation of $K$ (containing the inequality $N-\sum x_i ^2 \geq
0$ expressing that $K:=K_S$ is bounded, for some $N\in \mathbb{N}$),
any polynomial $f>0$ on $K_S$ belongs to $M_S$. The above results
have direct applications to the moment problem for semialgebraic
sets. Given a closed set $K\subseteq\reals^n$, the $K$-moment
problem is the question of when a linear functional
$\ell:\rx\rightarrow\reals$ is representable as integration with
respect to a positive Borel measure on $K$. Denoting the set of all
nonnegative polynomials on $K$ by $\Pos(K)$, a necessary condition
is that $\ell(f)\ge0$, for $f\in \Pos(K)$. In \cite{Hav1,Hav2},
Haviland proved that this necessary condition is sufficient. However
in general $\Pos(K)$ is not finitely generated,
so Haviland's result may be impractical. Fortunately, it
follows from Putinar's Archimedean Positivstellensatz that
nonnegativity of $\ell$ on $\Pos(K_S)$ is ensured once nonnegativity
of $\ell$ on the {\it finitely generated} archimedean quadratic
module $M_S$ (Definition \ref{preordering}) is established. Thus one is reduced to checking $s+2$
many systems of inequalities: 
\begin{equation}\label{archineq}
\begin{array}{ll}
	\ell(h^2 g_i)\ge0 & \mbox{ for }h\in\rx,\>i=0,\dots,s+1,\\
	g_0:=1, & g_{s+1}:=(N-\sum x_i ^2).
\end{array}
\end{equation}
Thus the $K_S$ -
moment problem is ``finitely solvable''. This can be summarized in a
topological statement: If
$\Pos(K_S)\subseteq\overline{M_S}^{\varphi}$, then the $K_S$- moment
problem for $M_S$ is finitely solvable. Here, $\varphi$ denotes the
finest locally convex topology on $\rx$, and
$\overline{M_S}^{\varphi}$ denotes the closure with respect to that
topology. In \cite{BCR, BCRBK, JLN} coarser topologies are considered: it is
shown that $\sos$ is dense in $\Pos([-1,1]^n)$ for the $\ell_1$
topology on $\rx$. In the language of moments, this result means
that every $\ell_1$ continuous linear functional, nonnegative on
$\sos$ (i.e. a {\it positive semidefinite functional}), is
representable by a positive Borel measure on $[-1,1]^n$.

In this paper, we generalize this setting to arbitrary locally
convex topologies on $\rx$. In Section \ref{SecPre}, we review
background on topological vector spaces.  In Section
\ref{KMPsec}, we present our setting as a threefold statement about
a locally convex topology $\tau$, a closed subset $K$ of $\reals^n$,
and a cone $C$ in $\rx$. We observe that if
$\Pos(K)\subseteq\overline{C}^{\tau}$ then any $\tau$-continuous
functional, nonnegative on $C$, is integration with respect to a
positive Borel measure on $K$ (see Proposition \ref{KMP}). We apply the above setting to the
example of the weighted norm-$p$ topologies, see Remark \ref{toz} and Theorems
\ref{WGLNK}, \ref{L1Case}, \ref{T:KMP-weight p-norm} and
\ref{T:KMP-weight p-norm2}. We compute the closure of the cone
$C:=\sos$ in the topology of coefficientwise convergence (see Proposition \ref{pwmp}).

Finally, we mention that this point of view has been recently
revisited: in \cite{JL:K-Moment} Lasserre proves that for a certain
fixed norm $\|\cdot\|_{w}$ defined in \cite{JL:K-Moment}, and any finite $S$,
$\overline{M_S}^{\|\cdot\|_{w}}=\Pos(K_S)$. Also, it follows from
recent work \cite{GMW} that Theorem \ref{GLNK} and Theorem
\ref{WGLNK} carry over with the cone $\sos$ replaced by the cone of
sums of $2d$-powers, $\sum\rx^{2d}$, for any integer $d\ge1$. In a forthcoming paper
\cite{G-K}, we compute the closure of $\sum\rx^{2d}$ with respect to
other locally convex topologies.

We wish to thank J. Cimpric, M. Marshall and V. Vinnikov for useful comments on preliminary versions of this paper.
\section{Preliminaries}\label{SecPre}
In this section, we give some background from functional analysis.
To make the paper self-contained and for the sake of the reader, we
have given a thorough explanation in this section.
\subsection{Background on Topological Vector Spaces}
In the following, all vector spaces are over the field of real numbers (unless otherwise specified).
A \textit{vector space topology} on a vector space $V$ is a topology $\tau$ on $V$ such that every point of $V$ is closed and the vector space operations, i.e., 
vector addition and scalar multiplication are $\tau$-continuous. A \textit{topological vector space} is a pair $(V,\tau)$ where $V$ is a vector space and $\tau$ is 
a vector space topology on $V$. A standard argument shows that $\tau$ is Hausdorff.

A subset $A\subseteq V$ is said to be \textit{convex} if for every $x,y\in A$ and $\lambda\in[0,1]$, $\lambda x+(1-\lambda)y\in A$.
A \textit{locally convex} topology is a vector space topology which admits a neighbourhood basis of convex open sets at each point.
A norm on $V$ is a function $\|\cdot\|:V\rightarrow\reals^{\ge0}$ satisfying
\begin{enumerate}
    \item{$\|v\|=0\Leftrightarrow x=0$,}
    \item{$\forall\lambda\in\reals,~\|\lambda v\|=|\lambda|\|v\|$,}
    \item{$\forall v_1,v_2\in V,~\|v_1+v_2\|\leq\|v_1\|+\|v_2\|$.}
\end{enumerate}
A topology $\tau$ on $V$ is said to be \textit{normable} (respectively \textit{metrizable}), if there exists a norm (respectively
metric) on $V$ which induces the same topology as $\tau$.
Every norm induces a locally convex metric topology on $V$ where the induced metric is defined by $d(v_1,v_2)=\|v_1-v_2\|$.
\begin{rem}
For two normed space $(X,\|\cdot\|)$ and $(Y,\|\cdot\|')$, a linear operator $T:X\rightarrow Y$ is said to be \textit{bounded} if there exists
$N\geq0$ such that for all $x\in X$, $\|Tx\|'\leq N\|x\|$. A standard result states that boundedness and continuity in normed
spaces are equivalent.
\end{rem}
For a topological vector space $(V, \tau)$ we denote the set of all continuous linear functionals $\ell:V\rightarrow\reals$ by $V^*$.
\begin{dfn}
For $C\subseteq V$, let 
\[
	C_{\tau}^{\vee}=\{\ell\in V^* : \ell\ge0\textrm{ on } C\}
\] 
be the \textit{first dual} of $C$ and define
the \textit{second dual} of $C$ by
\[
	C_{\tau}^{\vee\vee}=\{a\in V : \forall \ell\in C_{\tau}^{\vee},~\ell(a)\ge0\}.
\]
\end{dfn}
The following is immediate from the definition:
\begin{prop}\label{dulprop}
For a locally convex topological vector space $(V,\tau)$ and $C,D\subseteq V$ the following holds
\begin{enumerate}
    \item{$C\subseteq D\Rightarrow D_{\tau}^{\vee}\subseteq C_{\tau}^{\vee}$,}
    \item{$C\subseteq C_{\tau}^{\vee\vee}$,}
    \item{$C_{\tau}^{\vee\vee\vee}=C_{\tau}^{\vee}$.}
\end{enumerate}
\end{prop}
A subset $C$ of $V$ is called a \textit{cone} if $C+C\subseteq C$
and $\reals^+C\subseteq C$. It is clear that $C$ is convex.
\begin{thm}[Separation]\label{BST}
Suppose that $A$ and $B$ are disjoint nonempty convex sets in $V$. If $A$ is open, then there exists $\ell\in V^*$ and
$\gamma\in\reals$
such that $\ell(x)<\gamma\leq\ell(y)$ for every $x\in A$ and $y\in B$. Moreover, if $B$ is a cone, then $\gamma$ can be taken to be
$0$.
\end{thm}
\begin{proof}
For the first part, see \cite[Theorem 3.4]{RFA}. Suppose that $B$ is
a cone and suppose that $\ell$ and $\gamma$ are given by first part.
If $\gamma>0$, then $\ell(y)>0$ for all $y\in B$. Therefore
$\forall\epsilon>0~\epsilon y\in B$ so
$0<\gamma\leq\ell(\epsilon y)=\epsilon\ell(y)\xrightarrow{\epsilon\rightarrow0}0$.
This implies that $\gamma\leq0$.
Note also that $\ell\ge0$ on $B$. Otherwise, $\ell(x)<\gamma\leq\ell(y)<0$ for any $x\in A$ and some $y\in B$. Then for $r>0$,
$ry\in B$ and
$\ell(x)<\gamma\leq\ell(ry)=r\ell(y)\xrightarrow{r\rightarrow\infty}-\infty$
which is impossible. Therefore
$\ell(x)<\gamma\leq0\leq\ell(y)\quad\forall x\in A, \forall y\in B$.
Hence $\gamma$ can be chosen to be $0$.
\end{proof}
Below $\overline{C}^{\tau}$ denotes the closure of $C$ with respect to $\tau$. It follows that:
\begin{crl}[Duality]\label{duality}
For any nonempty cone $C$ in $(V,\tau)$, $\overline{C}^{\tau}=C_{\tau}^{\vee\vee}$.
\end{crl}
\begin{proof}
Since each $\ell\in C_{\tau}^{\vee}$ is continuous, for any $a\in\overline{C}^{\tau}$, $\ell(a)\ge0$, so $\overline{C}^{\tau}\subseteq C_{\tau}^{\vee\vee}$.
Conversely, if $a\not\in\overline{C}^{\tau}$ then since $\tau$ is locally convex, there exists an open convex set $U$ of $V$
containing $a$ with $U\cap C=\emptyset$. By \ref{BST}, there exists $\ell\in C_{\tau}^{\vee}$ such that $\ell(a)<0$, so $a\not\in C_{\tau}^{\vee\vee}$.
\end{proof}
\subsection{Finest Locally Convex Topology on $\rx$}
Let $V$ be any vector space over $\reals$ of countable infinite
dimension.
We define the direct limit topology
$\varphi$ on $V$ as follows: $U\subseteq V$ is open if and only if
$U\cap W$ is open in $W$ for each finite dimensional subspace $W$ of
$V$.
\begin{thm}
The open sets in $V$ which are convex form a basis for the direct limit topology.
Moreover $(V,\varphi)$ is a topological vector space and $\varphi$ is the finest locally convex topology on $V$.
\end{thm}
\begin{proof}
\cite[Section 3.6 and Theorem 3.6.1]{MPS}.
\end{proof}
\begin{rem}\label{NNT}
(i) The vector space $(V,\varphi)$ is not metrizable.
Let $U$ be a neighbourhood of $0$ in $V$. From the proof of \cite[Theorem 3.6.1]{MPS}, there exist
$a_i\in\reals^{>0}$, $i=1,2,\ldots$, such that $\prod_{i=1}^{\infty}\langle-a_i,a_i\rangle\subseteq U$, where
\[\prod_{i=1}^{\infty}\langle-a_i,a_i\rangle=\left\lbrace\sum_{i}t_ie_i:-a_i<t_i<a_i\right\rbrace,\]
and $\{e_i\}_{i=1}^{\infty}$ forms a basis for $V$ and all summands are
$0$ except for finitely many $i$.
If there exists a countable neighbourhood basis at $0$ then there exist real numbers $a_{ij}$, $i,j=1,2,\dots$ such that
$\prod_{i=1}^{\infty}\langle-a_{1i},a_{1i}\rangle, \prod_{i=1}^{\infty}\langle-a_{2i},a_{2i}\rangle, \dots$
forms a neighbourhood basis at $0$. Take $0<b_i<a_{ii}$ for each $i$, then $\prod_{i=1}^{\infty}\langle-b_i,b_i\rangle$ is a neighbourhood of $0$ which
does not contain any of the above basic open sets, a contradiction.

(ii) Every linear functional is continuous with respect to $\varphi$. For the weak topology (induced by the set of all linear functionals), convex sets have the same closure as they have under $\varphi$ \cite[Theorem 3.12]{RFA}.

(iii) Direct limit topology and finest locally convex topology are defined even when $V$ is uncountably infinite dimensional. But they only coincide when the
space is countable dimensional.
\end{rem}

\subsection{Moment Problem}
In analogy to the classical Riesz Representation Theorem, Haviland considered the problem
of representing linear functionals on
the algebra of polynomials by measures. The question of when, given a closed subset $K$ in $\reals^n$,
a linear map
$\ell:\rx\rightarrow\reals$ corresponds to a finite positive Borel measure $\mu$ on $K$ is known as
the Moment Problem.
\begin{dfn}
For a subset $K\subseteq\reals^n$, define the \textit{cone of nonnegative polynomials} on $K$ by
\[
\Pos(K)=\{f\in\rx~:~\forall x\in K~f(x)\ge0\}.
\]
\end{dfn}
\begin{thm}[Haviland]\label{haviland}
For a linear function $\ell:\rx\rightarrow\reals$ and a closed set $K\subseteq\reals^n$, the following
are equivalent:
\begin{enumerate}
	\item{There exists a positive regular Borel measure $\mu$ on $K$ such that, \[\forall f\in\rx \quad \ell(f)=\int_K f~d\mu.\]}
	\item{$\forall f\in \Pos(K)~\ell(f)\ge0$.}
\end{enumerate}
\end{thm}
The main challenge in applying Haviland's Theorem is verifying its
condition (2). We analyse this problem for a certain class of
closed subsets.
\begin{dfn}\label{preordering}
A subset $K\subseteq\reals^n$  is called a \textit{basic closed
semialgebraic} set if there exists a finite set of polynomials
$S=\{g_1,\ldots,g_s\}$ such that
$K=K_S:=\{x\in\reals^n~:~g_i(x)\ge0,~i=1,\ldots,s\}\>.$ Note that we
can define $K_C=\{x\in\reals^n~:~\forall f\in C ~ f(x)\ge0\}$
for any subset $C\subseteq\rx$, but this set may not be a semialgebraic
set. A subset $M$ of $\rx$ is called a \textit{quadratic module} if
$1\in M$, $M+M\subseteq M$, and for each $h\in\rx$, $h^2M\subseteq
M$. For $S=\{g_1,\ldots,g_s\}$, let
\[
    M_S:=\{\sum_{i=0}^s\sigma_ig_i : \sigma_i\in\sos\textrm{ for } i=0,\dots,s\textrm{ and } g_0=1\}.
\]
One can check that $M_S$ is the smallest quadratic module
of $\rx$ containing $S$. Clearly $M_S\subseteq\Pos(K_S)$.

A quadratic module $M$ of $\rx$ is said to be \textit{archimedean} if for a sufficiently large integer $N$, $(N-\sum x_i ^2)\in M$.
\end{dfn}
\begin{rem}
Note that to check whether a given linear functional
$\ell:\rx\rightarrow\reals$ is nonnegative on the quadratic module $M_S$, it suffices to
verify the following:
\begin{equation}\label{haveq}
	\ell(h^2 g_i)\ge0 \mbox{ for }h\in\rx,\>i=0,\dots,s,~ g_0:=1\>.
\end{equation}
If $\Pos(K_S)\subseteq(M_S)^{\vee\vee}_{\varphi}$, then, by Haviland's Theorem, every linear
functional nonnegative on $M_S$ corresponds to a measure on $K_S$.
Since $M_S$ is a cone in $\rx$ it follows by Corollary
\ref{duality} that $(M_S)^{\vee\vee}_{\varphi}=\overline{M}_S^{\varphi}$.
Therefore we are interested in the inclusion
\begin{equation}\label{SMP}
\Pos(K_S)\subseteq\overline{M}_S^{\varphi}.
\end{equation}
In other words, for a given basic closed semialgebraic set $K$, if
one can find a finite $S\subseteq\rx$ such that $K=K_S$ and at the
same time inclusion \eqref{SMP} holds, then the problem of
representing a functional by a measure on $K$ is reduced to
verifying that conditions \eqref{haveq} hold. \end{rem}

\section{The Moment Problem for functionals continuous with respect to a locally convex topology}\label{KMPsec}
We have seen that every linear functional is continuous with respect
to $\varphi$. We now consider a linear functional $\ell$, continuous
with respect to an arbitrary locally convex topology on $\rx$. We
further consider an arbitrary closed subset $K\subseteq\reals^n$,
and an arbitrary cone $C\subseteq\rx$.
\begin{prop}\label{KMP}
For a locally convex topology $\tau$ on $\rx$, a closed subset $K\subseteq\reals^n$ and
a cone $C\subseteq\rx$, the following are
equivalent:
\begin{enumerate}
    \item{$C_{\tau}^{\vee}\subseteq\Pos(K)_{\tau}^{\vee}$,}
    \item{$\Pos(K)\subseteq C_{\tau}^{\vee\vee}$,}
    \item{$\forall\ell\in C_{\tau}^{\vee}$ there exists a positive Borel measure on $K$ such that.
\[\forall f\in\rx~\ell(f)=\int_K f~d\mu.\]
    }
\end{enumerate}
\end{prop}
\begin{proof}
(1)$\Leftrightarrow$(2) is clear by Proposition \ref{dulprop}. For (1)$\Rightarrow$(3), note that
$\Pos(K)_{\tau}^{\vee}\subseteq\Pos(K)_{\varphi}^{\vee}$
then apply Haviland's Theorem \ref{haviland}. (3)$\Rightarrow$(1) is clear.
\end{proof}
\begin{rem}\label{toz} As explained in the introduction, given an arbitrary compact semialgebraic set $K_S$, 
to determine whether a linear functional comes from a measure on $K_S$, we need to check the $s+2$ conditions given in \eqref{archineq}.
In the next two subsection, we show that for certain compact convex
sets $K_S$ and functionals continuous in weighted norm $p$
topologies, we just need to check a single condition, namely
$\ell(h^2)\geq 0$ (see Corollary \ref{CtsPsd} and Theorems
\ref{WGLNK}, \ref{L1Case}, \ref{T:KMP-weight p-norm} and
\ref{T:KMP-weight p-norm2}).
\end{rem}
\subsection{Norm-$p$ Topologies}
We are interested in computing the closure of the cone $\sos$ in
$\rx$ under certain norm topologies. We start by reviewing some
basic facts about $\|\cdot\|_p$ norms (see \cite{Cnwy}). In what follows, $\naturals := \{0, 1,\cdots,\}$ the set of natural numbers, $\underline{X}
:=(X_1,\cdots, X_n)$, $\ux^{\alpha}:=X_1^{\alpha_1}\cdots
X_n^{\alpha_n}$ where $\alpha =(\alpha_1, \cdots, \alpha_n)$ and $|\alpha|=\alpha_1+\cdots+\alpha_n$.

Let $1\leq p<\infty$, and define the mapping $\|\cdot \|_p:\reals^{\naturals^n}\rightarrow\reals\cup\{\infty\}$ for each
$s:\naturals^n\rightarrow\reals$ with
$$\|s\|_p=(\sum_{\alpha\in\naturals^n}|s(\alpha)|^p)^{\frac{1}{p}}=
(\sum_{d=0}^{\infty}\sum_{|\alpha|=d}|s(\alpha)|^p)^{\frac{1}{p}}.$$
For $p=\infty$, define $\|s\|_{\infty}=\sup_{\alpha\in\naturals^n}|s(\alpha)|$.
For $1\leq p\leq\infty$, we let
\[
    \ell_p(\naturals^n)=\{s\in\reals^{\naturals^n}:\|s\|_p<\infty\},
\]
and
\[
    c_0(\naturals^n)=\{s\in\reals^{\naturals^n}:\lim_{\alpha\in\naturals^n}|s(\alpha)|=0\}.
\]
It is well-known that $\|\cdot\|_p$ is a norm on
$\ell_p(\naturals^n)$ and $(\ell_p(\naturals^n),\|\cdot\|_p)$ forms
a Banach space. Moreover, if $1\leq p<q\leq\infty$ then
$\ell_p(\naturals^n)\subsetneq\ell_q(\naturals^n)$. Let $V_p$ be the
set of all finite support real $n$-sequences, equipped with
$\|\cdot\|_p$. Fixing the monomial basis
$\{\ux^{\alpha}:\alpha\in\naturals^n\}$, we identify the space of
real polynomials $\rx$, endowed by $\|\cdot\|_p$-norm, with $V_p$.

For $1\leq p \leq \infty$, define the conjugate $q$ of $p$ as
follows:
\begin{itemize}
    \item{If $p=1$, let $q=\infty$},
    \item{If $p=\infty$, let $q=1$},
    \item{if $1<p< \infty$, let $q$ be the real number satisfying $\frac{1}{p}+\frac{1}{q}=1$.}
\end{itemize}
For a proof of Proposition \ref{denseness}, Lemma \ref{idmap} and Lemma \ref{HOLDER} see \cite{Cnwy} or \cite{RFA}.
\begin{prop}\label{denseness}
For $1\leq p<\infty$, $V_p$ is a dense subspace of $\ell_p(\naturals^n)$, and $V_{\infty}$ is dense in $c_0(\naturals^n)$.
A linear functional $\ell$ on $V_p$ is continuous if and only if $\|(\ell(\ux^{\alpha}))_{\alpha\in\naturals^n}\|_q<\infty$ where $q$
is the conjugate of $p$.
\end{prop}
\begin{lemma}\label{idmap}
For $1\leq p\leq q\leq\infty$, the identity map $id_{pq}:V_p\rightarrow V_q$ is continuous, i.e., $\|\cdot\|_p$ induces a
finer topology than $\|\cdot\|_q$ on $\rx$.
Therefore, for any $C\subseteq\rx$ we have $\overline{C}^{\|.\|_p}\subseteq\overline{C}^{\|.\|_q}$.
\end{lemma}
\begin{lemma}\label{HOLDER}
(H\"{o}lder's inequality) Let $1\leq p\leq \infty$ and let $q$ be the conjugate of $p$. Let
$a\in \ell_p(\naturals^n)$ and $b\in \ell_q(\naturals^n)$. Then $ab\in \ell_1(\naturals^n)$ and
$\|ab\|_1\leq\|a\|_p\|b\|_q$,
where we define $(ab)(\alpha):=a(\alpha) b(\alpha)$ for every
$\alpha \in \naturals^n$.
\end{lemma}
In \cite[Theorem 9.1]{BCR}, Berg, Christensen and Ressel showed that
the closure of $\sos$ in the $\|\cdot\|_1$-topology is
$\Pos([-1,1]^n)$. The proof given by Berg, Christensen and Ressel in
 \cite{BCR,BCRBK} is based on techniques from harmonic analysis on semigroups, whereas in \cite{JLN},
 Lasserre and Netzer produce a concrete sequence in $\sos$ converging to each $f\in\Pos([-1,1]^n)$ in $\|\cdot\|_1$. The above result of Berg \textit{et al.} easily extends to all $\|\cdot\|_p$-topologies, see Theorem \ref{GLNK} below. We first need the following.
\begin{prop}\label{evalcts}
Let $1\leq p\leq\infty$ and $\underline{x}\in\reals^n$, and let $e_{\underline{x}}:V_p\rightarrow\reals$ be the evaluation
homomorphism on $V_p$ defined by $e_{\underline{x}}(f):= f(\underline{x})$.
Then the following statements are equivalent:
\begin{enumerate}
	\item{$e_{\underline{x}}$ is continuous.}
	\item{$\|(\underline{x}^{\alpha})_{\alpha\in\naturals^n}\|_q<\infty$, where $q$ is the conjugate of $p$.}
	\item{$\underline{x}\in (-1,1)^n$ if $1\leq p<\infty$, and $\underline{x}\in [-1,1]^n$ if $p=\infty$.}
\end{enumerate}
\end{prop}
\begin{proof}
(2)$\Longleftrightarrow$(3) First assume that $1\leq p<\infty$. Let $\underline{x}=(x_1,\ldots,x_n) \in\reals^n$. Then
\[
    \begin{array}{rl}
        (\|(\underline{x}^{\alpha})_{\alpha\in\naturals^n}\|_p)^p & =\sum_{\alpha\in\naturals^n}|\underline{x}^{\alpha}|^{p}\\
         & =\sum_{\alpha_1,\cdots,\alpha_n=0}^{\infty}|x_1|^{p\alpha_1}\ldots|x_n|^{p\alpha_n}\\
         & =(\sum_{\alpha_1=0}^{\infty}|x_1|^{p\alpha_1})\cdots(\sum_{\alpha_n=0}^{\infty}|x_n|^{p\alpha_n}),\\
    \end{array}
\]
a product of geometric series. It follows that $\|(\underline{x}^{\alpha})_{\alpha\in\naturals^n}\|_p<\infty$ if and only if $|x_i|<1$ for $i=1,\ldots n$.
For $p=\infty$, $\|(\underline{x}^{\alpha})_{\alpha\in\naturals^n}\|_{\infty}=\sup_{\alpha\in\naturals^n}|\underline{x}^{\alpha}|$ is finite
if and only if  $|x_i|\leq 1$, for each $1\leq i\leq n$.

(2)$\Longrightarrow$(1) First suppose that $1\leq p<\infty$. Assuming $f(\ux)=\sum_{\alpha\in\naturals^n}f_{\alpha}\ux^{\alpha}$,
\[\begin{array}{rl}
\|e_{\underline{x}}\|=\sup_{\|f\|_p=1}|f(\underline{x})| & =\sup_{\|f\|_p=1}|\sum_{\alpha\in\naturals^n}f_{\alpha}\underline{x}^{\alpha}|\\
 & \leq\sup_{\|f\|_p=1}\sum|f_{\alpha}||\underline{x}^{\alpha}|\\
 \textrm{(By H\"{o}lder's inequality)} & \leq\sup_{\|f\|_p=1}\|f\|_p\|(\underline{x}^{\alpha})_{\alpha\in\naturals^n}\|_q\\
  & =\sup_{\|f\|_p=1}\|(\underline{x}^{\alpha})_{\alpha\in\naturals^n}\|_q\\
  & =\|(\underline{x}^{\alpha})_{\alpha\in\naturals^n}\|_q.
\end{array}\]
Therefore if $\|(\underline{x}^{\alpha})_{\alpha\in\naturals^n}\|_q<\infty$, the
$e_{\underline{x}}$ is continuous.

\noindent For $p=\infty$,
\[\begin{array}{rl}
 \|e_{\underline{x}}\|=\sup_{\|f\|_{\infty}=1}|f(\underline{x})| & =\sup_{\|f\|_{\infty}=1}|\sum_{\alpha\in\naturals^n}f_{\alpha}\underline{x}^{\alpha}|\\
 & \leq\sum_{\alpha\in\naturals^n}|f_{\alpha}|\cdot|\underline{x}^{\alpha}|\\
 & \leq\sum_{\alpha\in\naturals^n}|\underline{x}^{\alpha}|\\
 & =\|(\underline{x}^{\alpha})_{\alpha\in\naturals^n}\|_1,
\end{array}\]
So, if $\|(\underline{x}^{\alpha})_{\alpha\in\naturals^n}\|_1<\infty$, then $e_{\underline{x}}$ is continuous.\\
(1)$\Longrightarrow$(2) First consider the case where $1\leq
p<\infty$. Suppose that $e_{\underline{x}}$ is continuous on $V_p$.
By Proposition \ref{denseness}, $V_p$ is a dense subspace of
$(\ell_p(\naturals^n),\|\cdot\|_p)$ which is a Banach space.
Therefore $e_{\underline{x}}$ has a continuous extension to
$(\ell_p(\naturals^n),\|\cdot\|_p)$ denoted again by
$e_{\underline{x}}$. Using the fact that
$\ell_p(\naturals^n)^*=\ell_q(\naturals^n)$, continuity of
$e_{\underline{x}}$ is implies that
$\|(\underline{x}^{\alpha})_{\alpha\in\naturals^n}\|_q<\infty$.

Now suppose that $p=\infty$ and
$\|(\underline{x}^{\alpha})_{\alpha\in\naturals^n}\|_1=\infty$.
Then, by part (3), for some $1\leq i\leq n$, $|x_i|\geq1$. For any
$k\in\naturals$, $k\ge1$, let
$f_k(\underline{X})=\frac{1}{k}(1+X_i+X_i^2+\cdots+X_i^k)$ and
$g_k(\underline{X})=\frac{1}{k}(1-X_i+X_i^2-\cdots+(-X_i)^k)$.
Clearly $f_k,g_k\rightarrow0$ in $\|\cdot\|_{\infty}$, but
\[
\begin{array}{l}
    |e_{\underline{x}}(f_k)|\geq\frac{k+1}{k},\quad\textrm{ if }x_i\ge1,\\
    |e_{\underline{x}}(g_k)|\geq\frac{k+1}{k},\quad\textrm{ if }x_i\leq-1,
\end{array}
\]
Therefore in either cases at least one of $(e_{\underline{x}}(f_k))$ or $(e_{\underline{x}}(g_k))$ does not converge to $0$. Hence, for
$\underline{x}\not\in(-1,1)^n$, $e_{\underline{x}}$ is not continuous. This proves the result for $p=\infty$.
\end{proof}
\begin{crl}\label{closeness}
For $1\leq p\leq\infty$, $\Pos([-1,1]^n)$ is a closed subset of $V_p$.
\end{crl}
\begin{proof}
We first note that
\[
    \Pos([-1,1]^n)=\Pos((-1,1)^n)=\bigcap_{\underline{x}\in (-1,1)^n} e_{\underline{x}}^{-1}([0,+\infty)).
\]
However, by Proposition \ref{evalcts}(iii), for every $\underline{x}\in (-1,1)^n$, $e_{\underline{x}}$
is continuous on $V_p$. Hence the result follows.
\end{proof}
\begin{thm}\label{GLNK}
For $1\leq p \leq \infty$, $\overline{\sos}^{\|\cdot\|_p}=\Pos([-1,1]^n)$.
\end{thm}
\begin{proof}
Corollary \ref{closeness} implies that $\overline{\sos}^{\|\cdot\|_p}\subseteq \Pos([-1,1]^n)$.
The other inclusion follows from Lemma \ref{idmap} and \cite[Theorem 9.1]{BCR}.
\end{proof}
\begin{dfn}
Let $\ell : \rx\rightarrow\reals$ be a linear functional. We say that $\ell$ is \textit{positive
semidefinite} if $\ell(h^2)\ge0$ for every $h\in\rx$.
\end{dfn}

Proposition \ref{KMP}, Corollary \ref{duality} and Theorem \ref{GLNK} have an important consequence for the moment problem:

\begin{crl}\label{CtsPsd} Let $1\leq p\leq\infty$, and let
$\ell:\rx\rightarrow\reals$ be a linear functional on $\rx$ such
that $\|(\ell(\ux^{\alpha}))_{\alpha\in\naturals^n}\|_q<\infty$
where $q$ is the conjugate of $p$. If $\ell$ is positive
semidefinite, then there exists a positive Borel measure $\mu$ on
$[-1,1]^n$ such that $\forall f\in\rx\quad\ell(f)=\int_{[-1,1]^n}f~d\mu$.
\end{crl}

\subsection{Weighted Norm-$p$ Topologies}
We further extend Theorem \ref{GLNK} to
{\it weighted norm $p$-topologies}.
Let $r=(r_1,\ldots,r_n)$ be a $n$-tuple of positive real numbers and $1\leq p < \infty$. The vector space
\[\ell_{p,r}(\naturals^n):=\{s\in\reals^{\naturals^n}:\sum_{\alpha\in\naturals^n}|s(\alpha)|^p
r_1^{\alpha_1}\ldots r_n^{\alpha_n} <\infty\}\]
is a Banach space with respect to the norm
$\|s\|_{p,r}=(\sum_{\alpha\in\naturals^n}|s(\alpha)|^p r_1^{\alpha_1}\ldots r_n^{\alpha_n} )^{\frac{1}{p}}$.
Similarly,
\[\ell_{\infty,r} (\naturals^n):=\{s\in\reals^{\naturals^n}:\sup_{\alpha\in\naturals^n}|s(\alpha)|
r_1^{\alpha_1}\ldots r_n^{\alpha_n} <\infty\}\]
is a Banach space with respect to the norm
$\|s\|_{\infty,r}=\sup_{\alpha\in\naturals^n}|s(\alpha)| r_1^{\alpha_1}\ldots r_n^{\alpha_n}$.
Set
$$c_{0,r}(\naturals^n):=\{ s\in\reals^{\naturals^n}:\lim_{\alpha\in\naturals^n}|s(\alpha)|
r_1^{\alpha_1}\ldots r_n^{\alpha_n}=0 \}.$$ Then $c_{0,r}(\naturals^n)$ is a closed
subspace of $\ell_{\infty,r} (\naturals^n)$ with respect to the norm
$\|\cdot \|_{\infty,r}$.
Denote by $V_{p,r}$ the set of all finite support real
$n$-sequences (which we naturally identify with $\rx$), equipped with $\|\cdot\|_{p,r}$. Analogous to Proposition \ref{denseness}, the
completion of $V_{p,r}$ is $\ell_{p,r}(\naturals^n)$ when
$1\leq p < \infty$ and $c_{0,r}(\naturals^n)$ when $p=\infty$.
We now determine the continuous linear
functionals on $\ell_{p,r}(\naturals^n)$.
\begin{lemma}\label{L:dual-weighted}
Let $1\leq p \leq \infty$, and $q$ be the conjugate of $p$. \\Then $\ell_{1,r}(\naturals^n)^*=\ell_{\infty,r^{-1}}(\naturals^n)$, and $c_{0,r}(\naturals^n)^*=\ell_{1,r^{-1}}(\naturals^n)$.\\ For $1<p<\infty$, $\ell_{p,r}(\naturals^n)^*=\ell_{q,r^{-\frac{q}{p}}}(\naturals^n)$.
\end{lemma}
\begin{proof}
The map defined by
\[
\begin{array}{rl}
    T_{p,r}:\ell_p(\naturals^n)\longrightarrow & \ell_{p,r}(\naturals^n)\\
    (s(\alpha))_{\alpha\in\naturals^n}\longmapsto & (s(\alpha)r_1^{-\frac{\alpha_1}{p}}\ldots r_n^{-\frac{\alpha_n}{p}} )_{\alpha\in\naturals^n}
\end{array}
\]
is an isometric isomorphism with inverse
\[
\begin{array}{rl}
    T^{-1}_{p,r}: \ell_{p,r}(\naturals^n) \longrightarrow & \ell_p(\naturals^n)\\
    (t(\alpha))_{\alpha\in\naturals^n}\longmapsto & (t(\alpha)r_1^{\frac{\alpha_1}{p}}\ldots r_n^{\frac{\alpha_n}{p}})_{\alpha\in\naturals^n}
\end{array}
\]
Let $f\in\ell_{p,r}(\naturals^n)^*$. Then $f\circ T_{p,r}\in\ell_p(\naturals^n)^*=\ell_q(\naturals^n)$. Hence there exist $t\in \ell_q(\naturals^n)$
such that $t=f\circ T_{p,r}$.

Define the function $t' : \naturals^n \longrightarrow \reals$ by
$t'(\alpha)=r_1^{\frac{\alpha_1}{p}}\ldots r_n^{\frac{\alpha_n}{p}} t(\alpha)$, $\alpha \in \naturals^n$.
It is straightforward to verify that $t'\in \ell_{q,r^{-\frac{q}{p}}}(\naturals^n)$ if
$1\leq p < \infty$, and $t'\in \ell_{\infty,r^{-1}}(\naturals^n)$ if $p=1$.
Moreover $t'(\alpha)=f(\delta_{\alpha})$,
where $\delta_{\alpha}$ is the Kroneker function at the point $\alpha \in \naturals^n$.
The proof of $c_{0,r}(\naturals^n)^*=\ell_{1,r^{-1}}(\naturals^n)$ is similar.
\end{proof}
\begin{thm}\label{WGLNK}
Let $1\leq p \leq \infty$. Then:
\begin{enumerate}
	\item{For $1\leq p<\infty$, $\overline{\sos}^{\|\cdot\|_{p,r}}=\Pos(\prod_{i=1}^n[-r_i^{\frac{1}{p}},r_i^{\frac{1}{p}}])$.}
	\item{$\overline{\sos}^{\|\cdot\|_{\infty,r}}=\Pos(\prod_{i=1}^n[-r_i,r_i])$.}
\end{enumerate}
\end{thm}
\begin{proof}
(1) We first stablish $\Pos(\prod_{i=1}^n[-r_i^\frac{1}{p},r_i^\frac{1}{p}])\subseteq\overline{\sos}^{\|\cdot\|_{p,r}}$. Suppose that $f\in\rx$ and $f\geq0$ on $\prod_{i=1}^n[-r^\frac{1}{p}_i,r_i^\frac{1}{p}]$. Since the polynomial $\tilde{f}(\underline{X})=f(r_1^\frac{1}{p}X_1,\cdots,r_n^\frac{1}{p}X_n)$ is a nonnegative polynomial on $[-1,1]^n$, by Theorem \ref{GLNK}, there exist a sequence $(g_i)_{i\in\naturals}$ in $\sos$ which approaches to $\tilde{f}$ in $\|\cdot \|_p$. On the other hand, a straightforward computation obtains $\|g_i-\tilde{f}\|_p^p= \|\tilde{g_i}-f\|^p_{p,r}$,
where $\tilde{g}_i(\underline{X})=g_i(r_1^\frac{-1}{p}X_1 ,\ldots, r_n^\frac{-1}{p} X_n)$. Since $(\tilde{g}_i)_{i\in\naturals}$ is a
sequence in $\sos$, the result follows.  For the other inclusion, we first note that
\[
    \Pos(\prod_{i=1}^n[-r_i^\frac{1}{p},r_i^\frac{1}{p}])=\Pos(\prod_{i=1}^n (-r_i^\frac{1}{p},r_i^\frac{1}{p}))
=\bigcap_{\underline{x}\in \prod_{i=1}^n (-r_i^\frac{1}{p},r_i^\frac{1}{p})} e_{\underline{x}}^{-1}([0,+\infty)).
\]
A routine computation shows that for every $\underline{x} \in \prod_{i=1}^n(-r^\frac{1}{p}_i,r_i^\frac{1}{p})$,
\[
	(\underline{x}^{\alpha})_{\alpha\in\naturals^n}\in \ell_{\infty,r^{-1}}(\naturals^n) \ \ \text{if}\ \ p=1,
\]
and
\[
	(\underline{x}^{\alpha})_{\alpha\in\naturals^n}\in \ell_{q,r^\frac{-q}{p}}(\naturals^n) \ \ \text{if}\ \ 1<p<\infty,
\]
where $q$ is the conjugate of $p$.
It follows from Lemma \ref{L:dual-weighted} that $e_{\underline{x}}$ is continuous on $V_{p,r}$. Hence $\Pos(\prod_{i=1}^n[-r_i^\frac{1}{p},r_i^\frac{1}{p}])$ is a closed subset of $V_{p,r}$ containing $\sos$, and the inclusion follows.\\

(2) Again,
a routine computation shows that for every $\underline{x} \in \prod_{i=1}^n(-r_i,r_i)$,
$(\underline{x}^{\alpha})_{\alpha\in\naturals^n}\in \ell_{1,r^{-1}}(\naturals^n)$ and the rest follows similar to part (i).
\end{proof}
We now apply Theorem \ref{WGLNK} to obtain the $K$ moment
property for $\sos$ for certain convex compact polyhedron and
weighted norm-$p$ topologies as we summarize below in the following
three theorems:
\begin{thm}\label{L1Case}
Let $r=(r_1,\ldots,r_n)$ with $r_i>0$ for $i=1,\ldots,n$, and let $\ell:\rx\rightarrow\reals$ be a linear functional such that the
sequence $s(\alpha)=\ell(\ux^{\alpha})$ satisfies
\[
        \sup_{\alpha\in\naturals^n}|s(\alpha)|r_1^{-\alpha_1}\cdots r_n^{-\alpha_n}<\infty.
\]
Then $\ell$ is positive semidefinite if and only if there exists a positive Borel measure $\mu$ on $K=\prod_{i=1}^n[-r_i,r_i]$ such that
\[
    \forall f\in\rx  \quad \ell(f)=\int_Kf~d\mu.
\]
\end{thm}
\begin{thm}\label{T:KMP-weight p-norm}
Let $1< p< \infty$, $q$ the conjugate of $p$, and $r=(r_1,\ldots,r_n)$ with $r_i>0$ for $i=1,\ldots,n$. Suppose that
$\ell:\rx\rightarrow\reals$ is a linear functional such that the sequence $s(\alpha)=\ell(\ux^{\alpha})$ satisfies
\[
    \sum_{\alpha\in\naturals^n}|s(\alpha)|^qr_1^{-\frac{q}{p}\alpha_1}\cdots r_n^{-\frac{q}{p}\alpha_n}<\infty.
\]
Then $\ell$ is positive semidefinite if and only if there exists a positive Borel measure $\mu$ on
$K=\prod_{i=1}^n[-r_i^{\frac{1}{p}},r_i^{\frac{1}{p}}]$ such that
\[
    \forall f\in\rx \quad \ell(f)=\int_Kf~d\mu.
\]
\end{thm}
\begin{thm}\label{T:KMP-weight p-norm2}
Let $r=(r_1,\ldots,r_n)$ with $r_i>0$ for $i=1,\ldots,n$, and let $\ell:\rx\rightarrow\reals$ be a linear functional such that the
sequence $s(\alpha)=\ell(\ux^{\alpha})$ satisfies
\[
        \sum_{\alpha\in\naturals^n}|s(\alpha)|r_1^{-\alpha_1}\cdots r_n^{-\alpha_n}<\infty.
\]
Then $\ell$ is positive semidefinite if and only if there exists a positive Borel measure $\mu$ on $K=\prod_{i=1}^n[-r_i,r_i]$ such that
\[
    \forall f\in\rx  \quad \ell(f)=\int_Kf~d\mu.
\]
\end{thm}
\subsection{Closure of $\sos$ in the Topology of Coefficientwise Convergence}
In this final section, we characterize the closure of $\sos$ in the coefficientwise convergent topology.
A net $\{f_i\} \in \rx$ converges in the {\it  coefficientwise convergent topology} to $f\in \rx$
if for every $\alpha \in \naturals^n$, the coefficients of $\ux^\alpha$ in $f_i$ converges
to the coefficient of $\ux^\alpha$ in $f$.

\begin{prop}\label{pwmp}
For a polynomial $f(\ux)\in\rx$, the followings are equivalent:
\begin{enumerate}
	\item{$f(0)\ge0$.}
	\item{$f$ is coefficientwise limit of squares in $\rx$.}
	\item{$f$ is coefficientwise limit of elements of $\sos$.}
\end{enumerate}
\end{prop}
\begin{proof}
(2)$\Rightarrow$(3) and (3)$\Rightarrow$(1) are clear. It remains to prove (1)$\Rightarrow$(2). For each $i\ge1$ let $g_{i0}+g_{i1}+g_{i2}+\cdots$ be the 
power series expansion of $\sqrt{\frac{1}{i}+f}$ in the ring of formal power series $\rfps$ where $g_{ij}$ is a form of degree $j$ in $\rx$. 
Let $h_i=g_{i0}+g_{i1}+\cdots+g_{ii}$, then $h_i^2\xrightarrow{i\rightarrow\infty} f$ in the topology of coefficientwise convergence.
\end{proof}

\end{document}